\documentclass[a4,11pt]{amsart}
\usepackage[dvipdfmx]{graphicx}
\usepackage{amsmath}
\usepackage{amssymb}
\usepackage{subfigure}
\usepackage{stmaryrd}

\theoremstyle{plain}
\newtheorem{theorem}{Theorem}[section]
\newtheorem{lemma}[theorem]{Lemma}
\newtheorem{corollary}[theorem]{Corollary}
\newtheorem{proposition}[theorem]{Proposition}

\theoremstyle{definition}
\newtheorem{definition}{Definition}

\theoremstyle{remark}

\newtheorem*{notation}{Notation}

\providecommand{\real}{\mathbb{R}}
\providecommand{\convhull}[1]{\mathop{\rm conv}\nolimits #1}
\providecommand{\set}[2]{\left\{#1\, |\, #2\right\}}
\providecommand{\separg}[3]{\mathcal{H}_{#1 \mid #2}(#3)}
\providecommand{\notseparg}[3]{\mathcal{H}_{#1 \sim #2}(#3)}
\providecommand{\sep}[2]{\mathcal{C}_{#1 \mid #2}}

\title[On Partitioning Colored Points]{On Partitioning Colored Points}
\author{Takahisa Toda}
\address{Graduate School of Human and Environmental Studies, Kyoto University\\
  Yoshida-nihonmatsu-cho, Sakyo-ku, Kyoto 606-8501, JAPAN}
\email{toda.takahisa@hw3.ecs.kyoto-u.ac.jp}
\date{\today}

\begin{document}
\maketitle

\begin{abstract}
P.~Kirchberger proved that, for a finite subset $X$ of $\real^{d}$ such that each point in $X$ is painted with one of two colors, if every $d+2$ or fewer points in $X$ can be separated along the colors, then all the points in $X$ can be separated along the colors.
In this paper, we show a more colorful theorem.
\end{abstract}

\section{Introduction}\label{sec:intro}
Let us imagine that red candles and blue candles are placed on the top of a cake.
We want to cut it with a knife into two pieces in such a way that all the red candles are on one of the two pieces and all the blue candles are on the other piece.
A question is when we can cut it successfully.
The following Kirchberger's theorem~\cite{kirch:kirchberger} answers this question in a general setting.
Let $X$ be a finite subset of $\real^{d}$ consisting of red points and blue points.
We say that a subset $S$ of $X$ can be {\em separated along the colors} if there is a hyperplane $h$ such that all the red points in $S$ are contained in $h^{+}$ and all the blue points in $S$ are contained in $h^{-}$.
Here $h^{+}$ and $h^{-}$ are the two open halfspaces associated with $h$.
Kirchberger's theorem states that if every $d+2$ or fewer points in $X$ can be separated along the colors,
then all the points in $X$ can be separated along the colors.

We consider a more colorful cake cutting problem.
Suppose that we have a cake with candles each of which is painted with one of $k$ colors.
We want to cut it with a knife by several times in such a way that all candles with the same color are on one of the pieces, although candles with different colors must not be on the same piece.
Let us formally describe this setting.
Let $X$ be a finite subset of $\real^{d}$, and suppose that each point in $X$ is painted with one of $k$ colors.
We say that a subset $S$ of $X$ can be {\em partitioned by hyperplanes along the colors} if 
there is a family $\mathcal{F}$ of hyperplanes satisfying the following three conditions:
\begin{enumerate}
\item every hyperplane in $\mathcal{F}$ avoids the points in $S$;
\item for any two points in $S$ with different colors, there is a hyperplane in $\mathcal{F}$ separating them;
\item no hyperplane in $\mathcal{F}$ separates points in $S$ with the same color.
\end{enumerate}
Clearly this notion for $2$ colors is identical to the notion of separations.

\noindent
{\bf Our Contribution:} We prove the following theorem.
If every $(d+1)\cdot\eta_{d}(k)+k$ or fewer points in $X$ can be partitioned by hyperplanes along the colors,
then all the points in $X$ can be partitioned by hyperplanes along the colors, where $\eta_{d}(k)$ is given by
\[
\eta_{d}(k)=\sum_{i=0}^{d}\binom{k-2}{i}\; .
\]

\noindent
{\bf Related Work:} In~\cite{arocha:colorful}~\cite{attila:colorful}, they studied other colorful Kirchberger theorems.
They introduced the notion of separations for $k$ colors as follows.
Let $A_{i}$ be a finite set of points painted with the $i$-th color.
They say that $\bigcup_{i=1}^{k}A_{i}$ is separated if $\bigcap_{i=1}^{k}\convhull{A_{i}}=\emptyset$, where $\convhull{A_{i}}$ denotes the convex hull of $A_{i}$.
By the separation theorem, this notion for $2$ colors is identical to the notion of separations introduced by hyperplanes.
Their theorems are based on this notion.
We remark that our notion of partitions is essentially different from theirs.

In Section~\ref{sec:basic}, we introduce basic notions.
In Section~\ref{sec:arrangement}, we calculate the maximum cardinality of minimal transversals for the full subdivisions of $\mathcal{H}(X)$, which is essential to our main theorem.
In Section~\ref{sec:theorem}, we prove the main theorem.
In Section~\ref{sec:conclusion}, we put a conclusion.

\section{Basic Notions}\label{sec:basic}
Let $X$ be a nonempty set, and let $P$ be a collection of subsets of $X$.
The collection $P$ is called a {\em partition} of $X$ if it has the following three properties:
\begin{enumerate}
\item each member of $P$ is nonempty;
\item any two distinct members of $P$ are disjoint;
\item the union of all the members of $P$ is $X$.
\end{enumerate}
The set $X$ is called the {\em support} of this partition $P$.
Each member of $P$ is called a {\em component} of $P$.
In this paper we will only deal with partitions of a finite set.
Let $Y$ be a nonempty subset of $X$.
For a partition $P$ of $X$, we write $P|_{Y}$ for the {\em restriction} of $P$ to $Y$, i.e.
\[
P|_{Y} := \set{U\cap Y}{\text{$U$ is a component of $P$}}\setminus\{\emptyset\}.
\]
We say that the pair $(X, \mathcal{C})$ of a finite set $X$ and a collection of partitions of $X$ is a {\em division}.
By abuse of notation, a collection of partitions having a common support, say $\mathcal{C}$ itself, is called a division.
A subset of $\mathcal{C}$ is called a {\em subdivision} of $\mathcal{C}$.
Two elements $a$ and $b$ in $X$ are {\em separated} by a partition $P$ if 
$a$ and $b$ lie in different components of $P$.
A division $(X,\mathcal{C})$ is {\em full} if every two distinct elements in $X$ can be separated by some partition in $\mathcal{C}$.
In particular, a subset $\mathcal{S}$ of $\mathcal{C}$ is a {\em full subdivision} of $\mathcal{C}$ if $(X,\mathcal{S})$ is full.

We introduce a key concept in this paper.

\begin{definition}
Let $(X,\mathcal{C})$ be a full division.
A subset of $\mathcal{C}$ is called a {\em transversal} for the full subdivisions of $\mathcal{C}$ if 
it intersects all the full subdivisions of $\mathcal{C}$, that is, it contains at least one member from each full subdivision of $\mathcal{C}$.
\end{definition}

\begin{proposition}\label{prop:trans}
Let $(X,\mathcal{C})$ be a full division.
A minimal transversal for the full subdivisions of $\mathcal{C}$ has the form $\sep{a}{b}$ for some pair $(a,b)$ of elements in $X$, where $\sep{a}{b}$ denotes the set of all members of $\mathcal{C}$ that separate two elements $a$ and $b$ in $X$.
\end{proposition}

\begin{proof}
Let $\mathcal{T}$ be a minimal transversal for the full subdivisions of $\mathcal{C}$.
We prove that $\sep{a}{b}\subseteq\mathcal{T}$ for some pair $(a,b)$ of distinct elements in $X$.
Assume the opposite.
From this assumption, for each pair $(a,b)$ of distinct elements in $X$, there is a member of $\sep{a}{b}$ which does not belong to $\mathcal{T}$, which implies that every two distinct elements in $X$ can be separated by some member of $\mathcal{C}\setminus\mathcal{T}$.
Thus $\mathcal{C}\setminus\mathcal{T}$ is full and has no common member to $\mathcal{T}$.
This contradicts that $\mathcal{T}$ is a transversal.
Therefore $\mathcal{T}$ contains some $\sep{a}{b}$.

We observe that $\sep{a}{b}$ is a transversal for the full subdivisions of $\mathcal{C}$.
Indeed, by definition, for each full subdivision $\mathcal{F}$ of $\mathcal{C}$ there is a member of $\mathcal{F}$ separating $a$ and $b$.
This implies that $\mathcal{F}$ has a common member to $\sep{a}{b}$.

From the minimality of $\mathcal{T}$, we obtain that $\mathcal{T}=\sep{a}{b}$.
\end{proof}

In the remaining part of this section, we make a supplementary remark on a related concept to divisions.
A {\em separoid}~\cite{arocha:separoid} is a set together with a binary relation on its subsets, denoted by $\mid\;$, satisfying the following axioms:
\begin{enumerate}
\item $S\mid T\Rightarrow T\mid S$;
\item $S\mid T\Rightarrow S\cap T=\emptyset$;
\item $S\mid T$ and $U\subseteq S\Rightarrow U\mid T$.
\end{enumerate}
Given a division $(X,\mathcal{C})$, it induces a separoid over $X$ by declaring that $S\mid T$ if there is a member $P$ of $\mathcal{C}$ such that $S$ and $T$ are contained in different components of $P$.
Clearly $(X,\mathcal{C})$ can be recovered from the separoid provided that every member of $\mathcal{C}$ consists of at most two components.
In this paper, we shall be concerned with $\mathcal{C}$ itself rather than the separation relation $\mid$ induced by $\mathcal{C}$.

\section{Arrangement of points in $\real^{d}$}\label{sec:arrangement}
In this section, we will deal with a certain collection $\mathcal{H}(X)$ of partitions induced by hyperplanes in $\real^{d}$ and calculate the maximum cardinality of minimal transversals for the full subdivisions of $\mathcal{H}(X)$.

Let $X$ be a finite subset of $\real^{d}$.
We say that a partition $P$ of $X$ can be {\em realized by a hyperplane} if $P=\{X\}$ or there is a hyperplane $h$ such that $P=\{X\cap h^{+}, X\cap h^{-}\}$, where $h^{+}$ and $h^{-}$ are the two open halfspaces associated with $h$.

\begin{notation}
We denote by $\mathcal{H}(X)$ the set of all partitions of $X$ that can be realized by hyperplanes and by $\separg{a}{b}{X}$ the set of all members of $\mathcal{H}(X)$ that separate two elements $a$ and $b$ in $X$.
We define $\notseparg{a}{b}{X}:=\mathcal{H}(X)\setminus\separg{a}{b}{X}$, which is the set of all members of $\mathcal{H}(X)$ that do not separate $a$ and $b$.
\end{notation}

The following proposition is well-known (see \cite{harding:partitions}).

\begin{proposition}
Let $X$ be a set of $k$ points in $\real^{d}$.
The cardinality of $\mathcal{H}(X)$ is at most
\begin{equation*}
  \phi_{d}(k):=\sum_{i=0}^{d}\binom{k-1}{i}\; .
\end{equation*}
In particular, if the points are in general position, then the cardinality is exactly $\phi_{d}(k)$.
\end{proposition}

Recall that every minimal transversal for the full subdivisions of $\mathcal{H}(X)$ has the form $\separg{a}{b}{X}$ for some pair $(a,b)$.
We remark that $\separg{a}{b}{X}$ is a transversal but need not be minimal.
Suppose that three distinct points $a,b,$ and $c$ in $X$ are collinear and that $c$ lies between $a$ and $b$.
Then every member of $\mathcal{H}(X)$ separating $a$ and $c$ also separates $a$ and $b$, while every member of $\mathcal{H}(X)$ separating $c$ and $b$ does not separate $a$ and $c$.
Thus $\separg{a}{c}{X}$ is a proper subset of $\separg{a}{b}{X}$, which implies that $\separg{a}{b}{X}$ is not a minimal transversal.

\begin{proposition}
Let $X$ be a set of points in $\real^{d}$.
If no three points in $X$ are collinear, then every $\separg{a}{b}{X}$ for two distinct elements $a$ and $b$ in $X$ is a minimal transversal for the full subdivisions of $\mathcal{H}(X)$.
\end{proposition}

We omit the proof, which is straightforward.

\begin{notation}
Let $X$ be a set of $k\ (\geq 2)$ points in $\real^{d}$.
We denote by $\tau(X)$ the minimum cardinality of minimal transversals for the full subdivisions of $\mathcal{H}(X)$, and by $\tau_{d}(k)$ the minimum of all $\tau(X)$ for $k$-point sets $X$ whose points are in general position in $\real^{d}$.
\end{notation}

\begin{lemma}\label{lem:mintrans}
Let $X$ be a set of $k\ (\geq 2)$ points in general position in $\real^{d}$.
We have
\[
  \sum_{i=1}^{d}\binom{k-2}{i-1}\leq \tau(X)\; .
\]
\end{lemma}

\begin{proof}
Let $\mathcal{T}$ be a minimal transversal for the full subdivisions of $\mathcal{H}(X)$.
By Proposition~\ref{prop:trans}, $\mathcal{T}$ has the form $\separg{a}{b}{X}$ for some pair $(a,b)$.
Let us fix such two points $a$ and $b$, and define $X_{a}:=X\setminus\{a\}$.
For each member $P$ of $\mathcal{H}(X)$, its restriction $P|_{X_{a}}$ to $X_{a}$ can be realized by a hyperplane.
Conversely, every partition of $X_{a}$ realized by a hyperplane can be constructed in this way.
Thus we derive
\[
\mathcal{H}(X_{a})=\set{P|_{X_{a}}}{P\in\mathcal{H}(X)}.
\]

We observe that, for each member $R$ of $\mathcal{H}(X_{a})$, there are at most two members $P$ of $\mathcal{H}(X)$ satisfying $R=P|_{X_{a}}$.
Indeed, for $R=\{U_{1},U_{2}\}$, such candidates are described as $P:=\{U_{1}\cup\{a\}, U_{2}\}$ and $P':=\{U_{1},U_{2}\cup\{a\}\}$, if any.
Since the point $b$ lies in either $U_{1}$ or $U_{2}$, one of $P$ and $P'$ separates $a$ and $b$, and the other does not.
From this observation we derive
\[
|\mathcal{H}(X)|-|\mathcal{H}(X_{a})|\leq |\separg{a}{b}{X}|=|\mathcal{T}|.
\]
Since $X_{a}$ consists of points in general position, we derive
\begin{align*}
|\mathcal{H}(X)|-|\mathcal{H}(X_{a})| & =  \phi_{d}(k)-\phi_{d}(k-1) \\
& = \sum_{i=0}^{d} \left\{\binom{k-1}{i}- \binom{k-2}{i} \right\}\\
& = \sum_{i=1}^{d}\binom{k-2}{i-1}\; .
\end{align*}
Note that $\binom{k-1}{i}=\binom{k-2}{i}+\binom{k-2}{i-1}$ holds for $i\geq 1$:
this is clear in the case $k-1> i$, and the equation holds even in the case $k-1\leq i$ because $\binom{k-2}{i}$ becomes $0$.
\end{proof}

Figure~\ref{fig:mintrans} shows that the equality in the lemma above does not hold in general.

\begin{figure*}[tb]
  \begin{minipage}{0.33\hsize}
  \begin{center}
    \subfigure[]{\includegraphics[height=3cm]{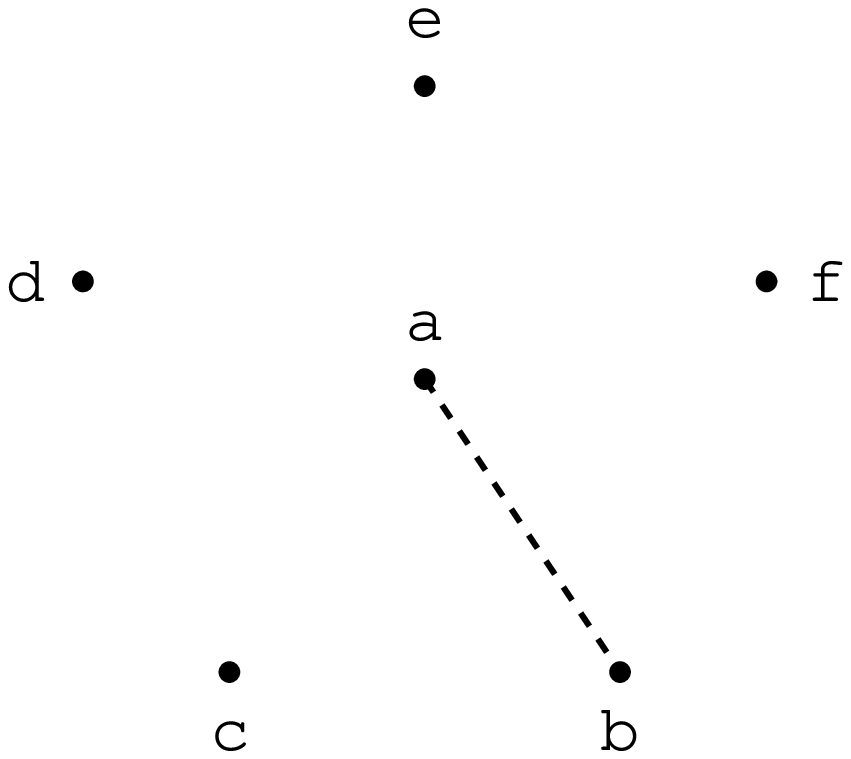}\label{fig:mintrans1}}
  \end{center}
  \end{minipage}
  \begin{minipage}{0.33\hsize}
  \begin{center}
    \subfigure[]{\includegraphics[height=3cm]{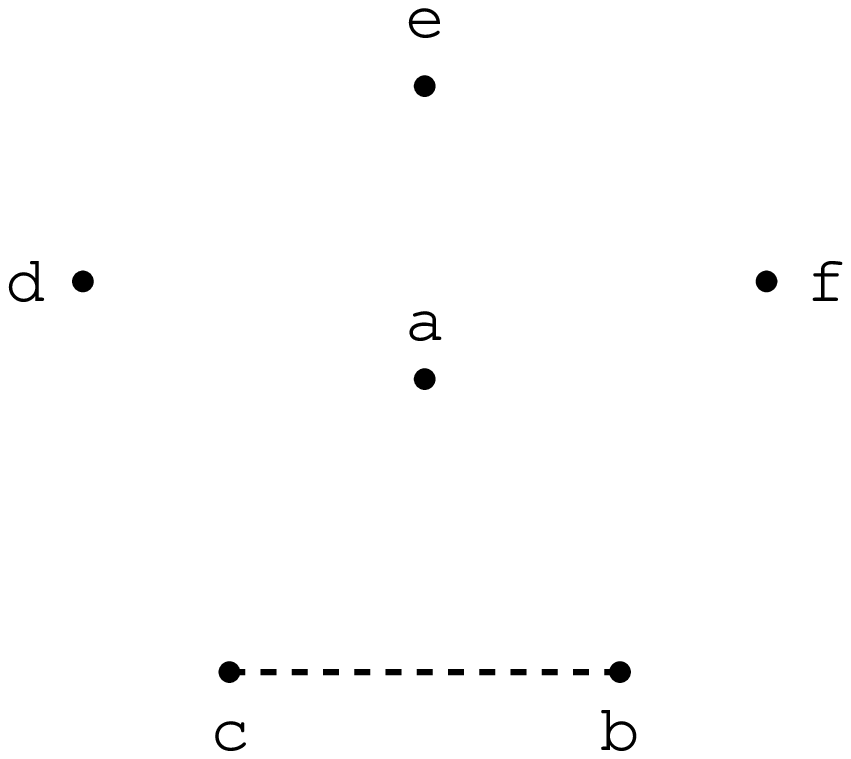}\label{fig:mintrans2}}
  \end{center}
  \end{minipage}
  \begin{minipage}{0.33\hsize}
  \begin{center}
    \subfigure[]{\includegraphics[height=3cm]{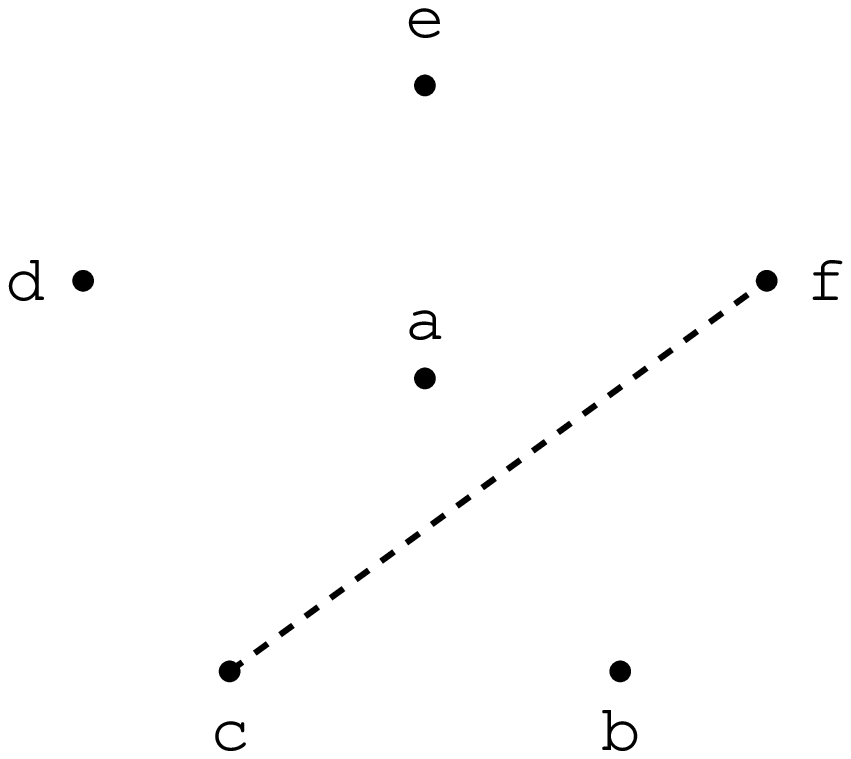}\label{fig:mintrans3}}
  \end{center}
  \end{minipage}
  \begin{minipage}{0.33\hsize}
  \begin{center}
    \subfigure[]{\includegraphics[height=3cm]{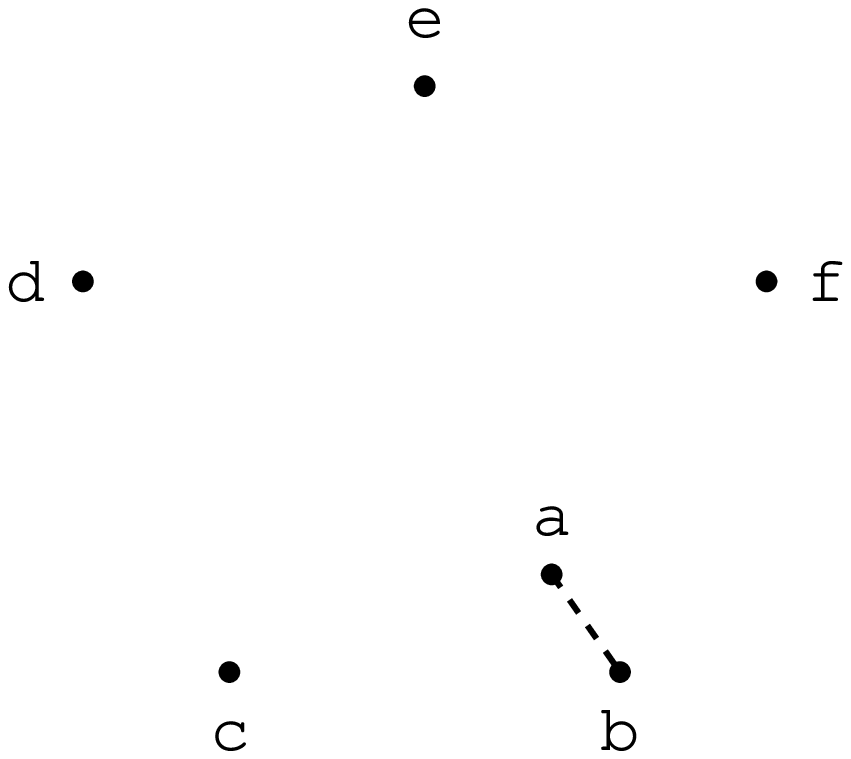}\label{fig:mintrans4}}
  \end{center}
  \end{minipage}
  \caption{Let $X$ be a set of six points in $\real^{2}$ consisting of the vertices $b,c,d,e,f$ of a regular pentagon and its center $a$.
In order to calculate $\tau(X)$, it suffices to deal with the three cases (a), (b), and (c) illustrated above.
The cardinalities of $\separg{a}{b}{X}$, $\separg{b}{c}{X}$, and $\separg{c}{f}{X}$ are $6$, $6$, and $10$, respectively.
Thus we obtain $\tau(X)=6$.
On the other hand, we have $\phi_{2}(6)-\phi_{2}(5)=5$.
If we move one of the six points, say $a$, closer to another point, say $b$, as illustrated in (d), then the cardinality of $\separg{a}{b}{X}$ becomes $5$, and the equality $\tau(X)=\phi_{2}(6)-\phi_{2}(5)$ holds in this case.}\label{fig:mintrans}
\end{figure*}

\begin{proposition}\label{prop:tau}
For $d\geq 1$ and $k\geq 2$, we have
\begin{equation*}
  \tau_{d}(k)=\sum_{i=1}^{d}\binom{k-2}{i-1}\; .
\end{equation*}
\end{proposition}

\begin{proof}
Let $X$ be a set of $k$ points in general position in $\real^{d}$.
Let us fix two distinct points $a$ and $b$ in $X$.
We show that if we move $a$ ``sufficiently closer'' to $b$ as shown in Figure~\ref{fig:mintrans}, then we obtain
\[|\separg{a}{b}{X}|=\sum_{i=1}^{d}\binom{k-2}{i-1}\;.\]
This is sufficient for the proof because the right-hand side is a lower bound for all $\tau(X)$ for $k$-point sets $X$ whose points are in general position in $\real^{d}$.

For each member $P$ of $\separg{a}{b}{X}$, let us choose a hyperplane $h_{P}$ which realizes $P$, and denote by $h_{P}^{+}$ the one of the two open halfspaces associated with $h_{P}$ which contains $b$.
Let us choose a point $c$ lying between $a$ and $b$ so that
$c$ lies in $h_{P}^{+}$ for all $P\in\separg{a}{b}{X}$ and the points in $(X\setminus\{a\})\cup\{c\}$ are in general position.
Clearly we can choose such a point.

Let us define $X':=(X\setminus\{a\})\cup\{c\}$.
By construction, $\separg{c}{b}{X'}$ has the following property:
for each member $P$ of $\separg{c}{b}{X'}$, there is a member $P'$ of $\notseparg{c}{b}{X'}$ satisfying $P|_{X'_{c}}= P'|_{X'_{c}}$, where $X'_{c}:=X'\setminus\{c\}$.
We can easily check this fact.
From the observation in the proof of Lemma~\ref{lem:mintrans}, we derive 
\begin{align*}
|\separg{c}{b}{X'}| & = |\mathcal{H}(X')|-|\mathcal{H}(X'_{c})|\\
& =\sum_{i=1}^{d}\binom{k-2}{i-1}\;.
\end{align*}
\end{proof}

\begin{notation}
Let $X$ be a set of $k\ (\geq 2)$ points in $\real^{d}$.
We denote by $\eta(X)$ the maximum cardinality of minimal transversals for the full subdivisions of $\mathcal{H}(X)$, and by $\eta_{d}(k)$ the maximum of all $\eta(X)$ for $k$-point sets $X$ whose points are in general position in $\real^{d}$.
\end{notation}

\begin{proposition}
For $d\geq 1$ and $k\geq 2$, we have
\[
\tau_{d}(k)+\eta_{d}(k)=\phi_{d}(k).
\]
\end{proposition}

\begin{proof}
Let $X$ be a set of $k$ points in general position in $\real^{d}$, and let $a$ and $b$ be two distinct points in $X$.
We will present another $k$-point set $X'$ such that 
$|\separg{a}{b}{X}|+|\separg{a'}{b'}{X'}|=\phi_{d}(k)$ for some pair $(a', b')$ of points in $X'$.

Let us fix a member $P$ of $\separg{a}{b}{X}$, and suppose that a hyperplane $h_{P}$ realizes $P$.
Let us construct the projective space extending $\real^{d}$.
In this space, we have the projective hyperplane $\widehat{h_{P}}$ extending $h_{P}$ and the complement of $\widehat{h_{P}}$ is a $d$-dimensional affine space.
We denote by $X'$ the set in this affine space corresponding to $X$, and by $a'$ and $b'$ the points corresponding to $a$ and $b$, respectively.

By construction, hyperplanes in the original space $\real^{d}$ which do not separate $a$ and $b$ correspond to those in the new affine space which separate $a'$ and $b'$.
This induces a bijection between $\notseparg{a}{b}{X}$ and $\separg{a'}{b'}{X'}$.
Thus we derive
\begin{align*}
\phi_{d}(k) & = |\separg{a}{b}{X}|+|\notseparg{a}{b}{X}|\\
           & =  |\separg{a}{b}{X}|+|\separg{a'}{b'}{X'}|.
\end{align*}
In particular, $|\separg{a}{b}{X}|=\tau_{d}(k)$ if and only if $|\separg{a'}{b'}{X'}|=\eta_{d}(k)$.
\end{proof}

\begin{corollary}
For $d\geq 1$ and $k\geq 2$, we have
\begin{equation*}
\eta_{d}(k)=\sum_{i=0}^{d}\binom{k-2}{i}\; .
\end{equation*}
\end{corollary}

We have calculated $\eta_{d}(k)$, which is the maximum of all $\eta(X)$ for $k$-point sets $X$ whose points are in general position in $\real^{d}$.
Regardless of the condition that the points are in general position, the fact remains that $\eta_{d}(k)$ gives the maximum cardinality.

\begin{proposition}\label{prop:eta}
Let $X$ be a set of $k\ (\geq 2)$ points in $\real^{d}$.
The maximum cardinality of minimal transversals for the full subdivisions of $\mathcal{H}(X)$ is at most $\eta_{d}(k)$.
\end{proposition}

\begin{proof}
By perturbation of the points in $X$, 
we can construct a set $X'$ of points in general position so that
$\mathcal{H}(X)$ is ``embedded'' in $\mathcal{H}(X')$ through the bijection $\phi\colon X\to X'$ induced by the perturbation.
Here we say that $\mathcal{H}(X)$ is embedded in $\mathcal{H}(X')$ through $\phi$ if, for each member $P$ of $\mathcal{H}(X)$, the partition $\set{\phi(U)}{U\in P}$ mapped by $\phi$ belongs to $\mathcal{H}(X')$ and this mapping is injective.
Thus we derive $\eta(X)\leq \eta(X')\leq \eta_{d}(k)$.
\end{proof}

\section{Main Theorem}\label{sec:theorem}
The following lemma slightly strengthens Kirchberger's theorem~\cite{kirch:kirchberger}.

\begin{lemma}\label{lem:kirch}
Let $X$ be a finite subset of $\real^{d}$ such that each point in $X$ is painted with one of two colors.
Let $p$ be a point in $X$.
If every subset of $X$ having $p$ whose size is at most $d+2$ can be separated along the colors, then $X$ can be separated along the colors.
\end{lemma}

\begin{proof}
Without loss of generality, we assume that 
$X$ consists of red points and blue points, and that $p$ is the origin in $\real^{d}$ painted with red.
For each point $a=(a_{1},\ldots,a_{d})\in X\setminus\{p\}$, we define the following halfspace $H_{a}$:
if it is red,
\[
H_{a}=\set{(\lambda_{1},\ldots,\lambda_{d})\in\real^{d}}{ \sum_{1}^{d}\lambda_{i} a_{i} < 1};
\]
otherwise,
\[
H_{a}=\set{(\lambda_{1},\ldots,\lambda_{d})\in\real^{d}}{ \sum_{1}^{d}\lambda_{i} a_{i} > 1}.
\]

Let $h$ be arbitrary hyperplane avoiding both of $a$ and $p$, which is denoted by $\sum_{i=1}^{d}\lambda_{i}x_{i}=1$.
Then we observe that $h$ separates $a$ and $p$ if and only if $\sum_{i=1}^{d}\lambda_{i}a_{i}>1$.
Thus a subset $S$ of $X$ having $p$ can be separated along the colors if and only if $\bigcap_{a\in S\setminus\{p\}}H_{a}\not=\emptyset$.

Suppose that every subset of $X$ having $p$ whose size is at most $d+2$ can be separated along the colors.
From the observation above, it follows that every $d+1$ or fewer members of $\{H_{a}\}_{a\in X\setminus\{p\}}$ have a common point.
From Helly's theorem (see \cite{danzer:convex}), we derive that all the members of $\{H_{a}\}_{a\in X\setminus\{p\}}$ have a common point.
Therefore there is a hyperplane separating $X$ along the colors.
\end{proof}

We are now ready to prove our main theorem.

\begin{theorem}\label{theo:main}
Let $X$ be a finite subset of $\real^{d}$, and suppose that each point in $X$ is painted with one of $k$ colors.
If every $(d+1)\cdot\eta_{d}(k)+k$ or fewer points in $X$ can be partitioned by hyperplanes along the colors, then all the points in $X$ can be partitioned by hyperplanes along the colors.
\end{theorem}

\begin{proof}
From each color, let us choose a point in $X$ painted with the color, and let us denote by $Y$ the set of the points chosen.
For each member $P$ of $\mathcal{H}(Y)$, let us construct a partition $\widehat{P}$ of $X$ in such a way that, for each component $U$ of $P$, the component $\widehat{U}$ of $\widehat{P}$ consists of points each of which has a common color to some point in $U$.
Clearly this extension is unique.

To prove the contraposition, assume that $X$ can not be partitioned by hyperplanes along the colors.
Let us define:
\[
\mathcal{T} = \set{P\in\mathcal{H}(Y)}{\widehat{P}\not\in\mathcal{H}(X)}.
\]
We show that $\mathcal{T}$ is a transversal for the full subdivisions of $\mathcal{H}(Y)$.
Otherwise, there is a full subdivision $\mathcal{F}$ of $\mathcal{H}(Y)$ which is disjoint to $\mathcal{T}$.
We observe that $\widehat{\mathcal{F}}:=\set{\widehat{P}}{P\in\mathcal{F}}$ consists of partitions of $X$ which can be realized by hyperplanes.
By construction, such hyperplanes realizing members of $\widehat{\mathcal{F}}$ partition $X$ along the colors.
This contradicts to our assumption.

If $\mathcal{T}$ is not a minimal transversal, we redefine $\mathcal{T}$ as a minimal one contained in $\mathcal{T}$.
For each member $P$ of $\mathcal{T}$, let us choose a subset $S_{P}$ of $X$ having a common point to $Y$ such that $S_{P}$ consists of at most $d+2$ points and 
it can not be separated along $\widehat{P}$.
By Lemma~\ref{lem:kirch}, there is such a set.
Then the set $Y\cup (\bigcup_{P\in\mathcal{T}}S_{P})$, which consists of at most $(d+1)\cdot \eta_{d}(k)+k$ points, can not be partitioned by hyperplanes along the colors.
\end{proof}

\section{Conclusion}\label{sec:conclusion}
We introduced the notion of partitions of colored points and proved a colorful Kirchgerber theorem with respect to this notion.
It is remained as a future work to improve the number $(d+1)\cdot\eta_{d}(k)+k$ in the statement of the main theorem.

\section*{Acknowledgments}
The author is grateful to Professor Hideki Tsuiki for his comments.

\end{document}